\documentclass[11 pt]{article}

\usepackage{amsfonts,amscd,amsthm}
\usepackage{amsmath,amssymb,amsfonts}
\usepackage{amsxtra, amscd, eucal}
\usepackage{caption} 
\usepackage{array}
\usepackage{graphicx, tikz}
\usepackage{mathrsfs} 
\usepackage{enumerate}
\usepackage{bm} 

\newtheorem{theorem}{Theorem}[section]
\newtheorem{lemma}[theorem]{Lemma}
\newtheorem{proposition}[theorem]{Proposition}
\newtheorem{corollary}[theorem]{Corollary}

\theoremstyle{definition}
\newtheorem{remark}[theorem]{Remark}

\newtheorem{case-}{Case}

\newtheorem{case.}{Case}

\newtheorem{example}[theorem]{Example}
\newtheorem{definition}[theorem]{Definition}
\DeclareMathOperator{\Tor}{Tor}
\DeclareMathOperator{\reg}{reg}
\DeclareMathOperator{\lcm}{lcm}
\DeclareMathOperator{\pd}{pd}
\DeclareMathOperator{\facets}{Facets}
\DeclareMathOperator{\taylor}{Taylor}
\DeclareMathOperator{\m}{\mathbf{m}}

\usepackage{xcolor}

\title{Betti numbers of monomial ideals via facet covers}
\author{Nursel Erey\footnote{E-mail: nurselerey@gmail.com} \and Sara Faridi\footnote{E-mail: faridi@mathstat.dal.ca (Research supported by an NSERC Discovery Grant)}}
    \date{\footnotesize Department of Mathematics and Statistics\\ Dalhousie University \\ Halifax, Nova Scotia, Canada B3H 3J5 \\[\baselineskip] }

\begin{document}
\maketitle


\begin{abstract}
We give a combinatorial condition that ensures a monomial ideal has
a nonzero Betti number in a given multidegree. As a result, we
combinatorially characterize all multigraded Betti numbers, projective
dimension and regularity of facet ideals of simplicial forests.  Our
condition is expressed in terms of minimal facet covers of simplicial
complexes.

\end{abstract}

\section{Introduction}

We begin by introducing the problem.  For a monomial ideal
$I=(m_1,\ldots ,m_q)$ one way to find Betti numbers of its minimal free
resolution is by calculating simplicial homology of subcomplexes of
its Taylor complex $\taylor(I)$, which is a simplex whose vertices are
labeled with the monomials $m_1, \ldots ,m_q$ and whose each face is
labeled with the $\lcm$ of the monomials labeling the vertices of that
face. As an example consider $I=(xy,yz,xz)$. In this case $\taylor(I)$
is the labeled simplex on the left in Figure~\ref{f:Taylor}.

\begin{figure}[htp]
\centering
\begin{tikzpicture}
[scale=1.00,vertices/.style={draw, fill=black, circle, inner sep=0.5pt}]

\filldraw[fill=black!25, draw=black] (-1.1,0) -- (0,1.1) -- (1.1, 0) -- cycle; 

\node[vertices, label=left:{$xy$}] (a) at (-1.1,0) {};
\node[vertices, label=left:{$yz$}] (b) at (0,1.1) {};
\node[vertices, label=right:{$xz$}] (c) at (1.1,0) {};
\node[label=below:{$xyz$}](e) at (0,.8){}; 
\node[label=left:{$xyz$}](f) at (-.5,.5){}; 
\node[label=right:{$xyz$}](g) at (.5,.5){}; 
\node[label=below:{$xyz$}](h) at (0,.1){}; 
\node[vertices, label=left:{$xy$}] (i) at (3.9,0) {};
\node[vertices, label=left:{$yz$}] (j) at (5,1.1) {};
\node[vertices, label=right:{$xz$}] (k) at (6.1,0) {};

\foreach \to/\from in {a/b, b/c, c/a}
\draw [][-] (\to)--(\from);

\end{tikzpicture}
\caption{}
\label{f:Taylor}
\end{figure}

According to a formula of Bayer, Peeva, and Sturmfels \cite{BPS},
$b_{i,xyz}(I)$ (the $i$th multigraded Betti number of $I$ supported on
the monomial $xyz$) is the dimension of the $(i-1)$st reduced homology
module of $\taylor(I)_{<xyz}$, which is the subcomplex of
$\taylor(I)$ consisting of faces whose labels strictly divide $xyz$
(pictured on the right in Figure~\ref{f:Taylor}).

It therefore follows that $b_{1,xyz}(I)=2$, and in fact, since this is
the only possible degree three monomial in the variables $x,y,z$, we
see that $b_{1,3}(I)=2$.

The formula of Bayer, Peeva, and Sturmfels works for all multigraded
Betti numbers, but in practice it is enough to focus on the top
multidegree $\m$, that is, the product of all the variables appearing in
the generators of the ideal~(\cite[Lemma~3.1]{erey faridi}).
So the question of finding Betti numbers of $I$ supported on $\m$ is
equivalent to the problem of computing the simplicial homology of
$\taylor(I)_{< \m}$.

Suppose now we take a minimal subset
$\sigma=\{m_{a_1},\ldots,m_{a_p}\}$ of the minimal generating set of
$I$, such that $\lcm(m_{a_1},\ldots,m_{a_p})=\m$. Since $\sigma$  is
minimal with this property, we have $$\sigma \notin \taylor(I)_{<\m}
\mbox{ and } \tau \in \taylor(I)_{<\m} \mbox{ for every } \tau
\subsetneq \sigma.$$

Thus the boundary of the face $\sigma$ is a candidate for a homological
cycle in $\taylor(I)_{<\m}$, and as a result every time we have such a
$\sigma$ we might have a nonvanishing Betti number.

In our running example in Figure~\ref{f:Taylor}, we can take
$\sigma=\{xy,yz\}$ and we see that its boundary, consisting of the
vertices labeled $xy$ and $yz$, is indeed a homological $0$-cycle, and
hence $b_{1,xyz}(I)\neq 0$.

The premise of this paper is exploring which subsets $\sigma$ of
the generators of $I$ do in fact ensure a nonvanishing Betti
number. This is done by translating such subsets into ``facet covers''
of simplicial complexes. 

In this paper we introduce ``well ordered facet covers'' and show that:

\begin{itemize} 

\item For a monomial ideal $I$, the existence of a well ordered facet
  cover of cardinality $i\geq 1$ implies that $b_{i}(S/I) \neq 0$
  (Corollary~\ref{cor: nonvanishing Betti number sufficient
    condition}).

\item If $I$ is the facet ideal of a simplicial forest, then the
  existence of a well ordered facet cover of cardinality $i\geq 1$ is
  equivalent to $b_{i}(S/I) \neq 0$ (Theorem~\ref{thm: characterization of Betti numbers}).

\item If $I$ is the facet ideal of a simplicial forest, then the regularity and projective dimension 
  of $I$ can be combinatorially described in terms of well ordered facet covers (Corollary~\ref{c:reg-pd}).
\end{itemize}

Well ordered facet covers generalize the notion of strongly disjoint
bouquets introduced by Kimura \cite{kimura} and our condition
generalizes the main result in~\cite{kimura} from graphs to simplicial
complexes.  Our paper fits within the recent interest in
combinatorially bounding or computing homological invariants of
squarefree monomial ideals, see for example \cite{bouchat ha okeefe,
  erey faridi, ha regularity, ha van tuyl, ha van tuyl splitting, he
  van tuyl, katzman, kimura, kummini, mahmoudi et al, morey
  villarreal, nevo, woodroofe, zheng}.

The last section of the paper is dedicated to the study of
well ordered facet covers of graphs.

%
\section{Background}
\subsection{Simplicial complexes}

A \textbf{simplicial complex} $\Gamma$ on a finite vertex set $V(\Gamma)$ is a set of subsets of $V(\Gamma)$ such that $\{v\}\in \Gamma$ for every $v\in V(\Gamma)$ and if $F\in \Gamma$, then $G\in \Gamma$ for every $G\subseteq F$. The elements of $\Gamma$ are called \textbf{faces} and maximal faces with respect to inclusion are called \textbf{facets}. 
If $F_1,\ldots,F_q$ are all the facets of $\Gamma$, then we write $\Gamma=\langle F_1,\ldots,F_q \rangle$ and say $\Gamma$ is generated by $F_1,\ldots,F_q$. Also we write $\facets(\Gamma)$ for the set of facets of $\Gamma$. If $F$ is a facet of $\Gamma$, then $\Gamma\setminus \langle F\rangle$ denotes the simplicial complex whose facet set is $\facets(\Gamma)\setminus \{F\}$. A \textbf{subcollection} of $\Gamma$ is a simplicial complex $\Delta$ such that every facet of $\Delta$ is also a facet of $\Gamma$. If $A$ is a set of vertices of $\Gamma$, then the \textbf{induced subcollection} $\Gamma_A$ is the simplicial complex $\langle F\in \facets(\Gamma) \mid F\subseteq A \rangle $. 

A simplicial complex $\Gamma$ is \textbf{connected} if for any two facets $F$ and $G$ of $\Gamma$ there exists facets $F_0=F,F_1,\ldots,F_k=G$ of $\Gamma$ such that $F_i\cap F_{i+1}\neq \emptyset$ for every $i=0, \ldots,k-1$.

A facet $F$ of $\Gamma$ is called a \textbf{leaf} if either $F$ is the only
facet of $\Gamma$, or there exists a facet $G\in \Gamma$ such that
$G \neq F$ and $F\cap H \subseteq G$ for every facet $H\neq F$. By definition,
every leaf $F$ of $\Gamma$ contains a \textbf{free vertex}, i.e., a vertex $v$ such that $v \notin
H$ for every facet $H\in \facets(\Gamma)\setminus\{F\}$. A simplicial complex $\Gamma$ is called a \textbf{simplicial forest} if every nonempty subcollection of $\Gamma$ has a
leaf. Moreover, if $\Gamma$ is connected, then we say $\Gamma$ is a \textbf{simplicial tree}.

\begin{figure}[htp]
\centering
\begin{tikzpicture}
[scale=1.00,vertices/.style={draw, fill=black, circle, inner sep=0.5pt}]

\filldraw[fill=black!25, draw=black] (-2,1) -- (-3,-1) -- (0, -1) -- cycle; 
\filldraw[fill=black!25, draw=black] (-2,1) -- (0,-1) -- (0, 1) -- cycle; 
\filldraw[fill=black!25, draw=black] (2,1) -- (0,-1) -- (0, 1) -- cycle; 
\filldraw[fill=black!25, draw=black] (2,1) -- (0,-1) -- (3, -1) -- cycle; 

\node[vertices, label=left:{$x_1$}] (b) at (-2,1) {};
\node[vertices, label=left:{$x_2$}] (c) at (-3,-1) {};
\node[vertices, label=below:{$x_3$}] (d) at (0,-1) {};
\node[vertices, label=above:{$x_4$}] (e) at (0,1) {};
\node[label=above:{$F_1$}](k) at (0.6,0){}; 
\node[label=above:{$F_2$}](l) at (1.8,-0.8){}; 
\node[label=above:{$F_3$}](m) at (-1.8,-0.8){}; 
\node[label=above:{$F_4$}](n) at (-0.6,0.0){}; 

\node[vertices, label=right:{$x_5$}] (f) at (2,1) {};
\node[vertices, label=below:{$x_6$}] (h) at (3,-1) {};
\foreach \to/\from in {b/c, b/d, c/d, d/e, d/f, d/h, f/d, e/f, f/h}
\draw [][-] (\to)--(\from);
\end{tikzpicture}
\caption{A simplicial tree $\Gamma = \langle F_1, F_2, F_3, F_4 \rangle$ which has the facet ideal $\mathcal{F}(\Gamma)=(x_1x_2x_3, x_1x_3x_4, x_3x_4x_5, x_3x_5x_6 )$.}
\label{fig: simplicial tree}
\end{figure}
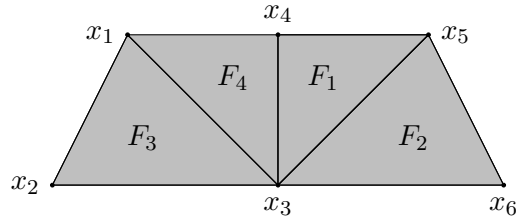

\subsection{Facet ideals}

A set of vertices $C$ of $\Gamma$ is called a \textbf{vertex cover} if
$F\cap C \neq \emptyset$ for every $F\in \facets(\Gamma)$.  A set $D
\subseteq \facets(\Gamma)$ is called a \textbf{facet cover} of
$\Gamma$ if every vertex $v$ of $\Gamma$ belongs to some $F$ in $D$. A
facet cover (respectively vertex cover) is called \textbf{minimal} if
no proper subset of it is a facet cover (respectively vertex cover) of
$\Gamma$.

A set $D$ of facets of $\Gamma$ is called a \textbf{matching} if the facets in $D$ are pairwise disjoint. Moreover, if $\facets(\Gamma_A)=D $ for $A= \displaystyle \cup_{F\in D}F $, then $D$ is called an \textbf{induced matching}. The maximum cardinality of an induced matching of $\Gamma$ is called the \textbf{induced matching number}.

Let $\Gamma$ be a simplicial complex on the vertices $x_1,\ldots, x_n$
and let $\Bbbk$ be a fixed field. The \textbf{facet ideal}  of $\Gamma$ is the ideal
$$ \mathcal{F}(\Gamma)=(x_{i_1}\cdots x_{i_k} \mid \{x_{i_1},\ldots,
x_{i_k}\} \text{ is a facet of } \Gamma)$$ of $S=\Bbbk[x_1,\ldots,x_n]$
generated by the monomials corresponding to the facets of $\Gamma$. 

If a monomial ideal $I$ is generated by the monomials $m_1,\ldots
,m_q$, then we write $I=(m_1,\ldots,m_q)$. One can set a one-to-one
correspondence between squarefree monomial ideals in
$S=\Bbbk[x_1,\ldots,x_n]$ and simplicial complexes on the vertex set
$\{x_1,\ldots,x_n\}$ via facet ideals. When it is convenient we shall
use a face $\{x_{i_1}, \ldots , x_{i_k}\}$ of a simplicial complex
$\Gamma$ interchangeably with the monomial $x_{i_1}\cdots x_{i_k}$
unless there is a confusion.

\begin{remark}[Localization of facet ideals] 
Suppose that $\Gamma$ is a simplicial complex on vertices $x_1, \ldots,
x_n$ and that $\{x_{i_1}, \ldots, x_{i_k}\}$ is a vertex cover of
$\Gamma$. Then $P=(x_{i_1}, \ldots, x_{i_k})$ is a prime containing the
ideal $\mathcal{F}(\Gamma)$, and by $\mathcal{F}(\Gamma)_P$ we mean the
ideal of $\Bbbk[x_1,\ldots ,x_n]$ whose monomial generators are
identified with those of the localized ideal of $\mathcal{F}(\Gamma)$ at $P$. We will make use of the fact that if $\Gamma$
is a simplicial forest, then $\mathcal{F}(\Gamma)_P$ is also the facet
ideal of a simplicial forest \cite{faridi cm properties}. For example,
if $\Gamma$ is the simplicial complex in Figure~\ref{fig: simplicial
  tree}, $P_1=(x_1, x_3, x_5)$ and $P_2=(x_2, x_4, x_6),$ then
$\mathcal{F}(\Gamma)_{P_1}=(x_1x_3, x_3x_5)$ and
$\mathcal{F}(\Gamma)_{P_2}=(x_2, x_4, x_6)$.
\end{remark}



\subsection{Free Resolutions}
Any monomial ideal $I$ of $S$ admits a minimal multigraded free resolution
\begin{equation}\label{eq:resolution}
\footnotesize 0 \longrightarrow \bigoplus_{\m \in \mathbb{N}^n}
 S(-\m)^{b_{r,\m}(I)} \overset{d_r}{\longrightarrow}
 \cdots\longrightarrow \bigoplus_{\m \in \mathbb{N}^n}
 S(-\m)^{b_{1,\m}(I)} \overset{d_1}{\longrightarrow}
 \bigoplus_{\m \in \mathbb{N}^n}
 S(-\m)^{b_{0,\m}(I)} \overset{d_0}{\longrightarrow} I
 \longrightarrow 0 \end{equation}
where $S(-\m)$ is the shifted free module whose
generator has multidegree $\m$ . The associated ranks
$b_{i,\m}(I)$ are called \textbf{multigraded Betti numbers} of
$I$. Similarly, the \textbf{graded Betti numbers} $b_{i,j}(I)$ are the
associated ranks of the free modules in the minimal graded free
resolution of $I$. The \textbf{projective dimension} of $I$ is
$\pd(I)=\max\{i \mid b_{i,j}(I)\neq 0 \text{ for some } j\}$ and the
\textbf{regularity} of $I$ is $\reg(I)=\max\{j-i \mid b_{i,j}(I)\neq 0
\}$.

When $S\neq I\neq 0$, the minimal free resolution of $S/I$ is the same as the one of $I$ described in~(\ref{eq:resolution}) with $I\to 0$ replaced with $S\to S/I \to 0$. As a result, we have $b_{i,\m}(S/I)=b_{i-1,\m}(I)$ for all
indices $i\geq 1$ and all monomials $\m$.

Due to polarization, resolutions of monomial ideals reduce
to resolutions of squarefree monomial ideals, allowing one to use
graphs and simplicial complexes as combinatorial tools.

It is well known that when $I$ is squarefree, the nonzero Betti
numbers lie only in squarefree multidegrees, see for example
\cite[Theorem~57.9]{peeva}. Therefore the graded Betti numbers of a
squarefree monomial ideal $I$ are given by $ b_{i,j}(I)=\sum
b_{i,\m}(I)$ where the sum is taken over all squarefree
monomials $\m$ such that $\deg(\m)=j$.

If $I$ is the facet ideal of a simplicial forest, then there are restrictions on the value and position of the multigraded Betti numbers as well.
 
\begin{theorem}[{\cite[Theorem~3.5]{erey faridi}}]\label{thm: multigraded betti numbers of simplicial forests} Let $\Gamma$ be a simplicial forest. Then multigraded Betti numbers 
of $\mathcal{F}(\Gamma)$ are either $0$ or $1$. Moreover, if for some
monomial $\m$ we have $b_{i,\m}(\mathcal{F}(\Gamma))\neq 0$, then
$b_{j,\m}(\mathcal{F}(\Gamma))= 0$ for all $j \neq i$.
\end{theorem}
The proof of the theorem above was based on the splitting formula of
H\`{a} and Van Tuyl given in \cite[Theorem~5.8]{ha van tuyl
  splitting}. When we focus on the top degree Betti numbers of a
simplicial tree $\Gamma$ (that is, the largest possible integer $q$ for which
the graded Betti number $b_{i,q}(\mathcal{F}(\Gamma))$ can be nonzero for some
$i$), this formula comes down to
\begin{equation}\label{eq: recursive multigraded formula}
b_{i,n}(\mathcal{F}(\Gamma))=b_{i-1,n-|F|}(\mathcal{F}(\Gamma\setminus \langle F \rangle)_{(x_i: \ x_i \notin F)})
\end{equation}
where $F$ is a leaf of $\Gamma$, $i\geq 1$ and $n$ is the number of
vertices of $\Gamma$. Note that by \cite[Lemma 4.5]{faridi cm
  properties}, if $\Gamma$ is a simplicial forest, then
$\mathcal{F}(\Gamma)_P$ is also the facet ideal of a simplicial
forest, making~\eqref{eq: recursive multigraded formula} a recursive
formula. We will use this equation in the sequel.

\subsubsection{The Lyubeznik Resolution}
Let $I$ be a monomial ideal of $S=\Bbbk[x_1,\ldots,x_n]$. In \cite{lyubeznik} Lyubeznik constructed an explicit free resolution
$$ \mathbf{L}:0 \rightarrow L^p \xrightarrow{d_p} L^{p-1} \xrightarrow{d_{p-1}} \cdots \xrightarrow{d_1}L^0\xrightarrow{} S/I \xrightarrow{} 0$$  of $S/I$. Although this resolution is not minimal in general, one can use it to compute the graded Betti numbers of $S/I$ since
$$b_{i,j}^S(S/I)= \dim _{\Bbbk}(\Tor_i^S(S/I, S/m)_j) =\dim _{\Bbbk}(H_i(\mathbf{L}\otimes_S S/m)_j)$$
where $m=(x_1,\ldots,x_n)$ is the irrelevant maximal ideal. The following presentation of Lyubeznik resolutions as simplicial resolutions is adopted from \cite{mermin}.
Let $M$ be the set of minimal generators of $I$. We fix a total ordering $m_1< m_2<\cdots<m_s$ on the elements of $M$. 
Let $\taylor(I)$ be the full simplex whose vertex set is $M$. We label every face $F$ of $\taylor(I)$ with
$$\lcm(F)=\lcm(m_i \mid m_i \in F) $$
the least common multiple of the vertices belonging to $F$.
  For a monomial $m\in I$, define
$$\min(m)= {\min}_{<}\{m_i\in M \mid m_i \text{ divides } m\} $$
and for a face $F\in \taylor(I)$ define
$$\min(F)=\min(\lcm(F)). $$
A face $F\in \taylor(I)$ is called \textbf{rooted} (or \textbf{$\bm{L}$-admissible}) if for every $\emptyset\neq G\subseteq F$, the property $\min(G)\in G$ holds. The rooted faces of $\taylor(I)$ form a simplicial complex $\Lambda_{I,<}$ which is called the \textbf{Lyubeznik simplicial complex} associated to $I$ and $<$. The Lyubeznik resolution of $S/I$ is the simplicial resolution supported on $\Lambda_{I,<}$. In other words, $L_i$ is the free $S$-module generated by $\{[F] : F\in \Lambda_{I,<} \text{ and } |F| = i\}$, where $[F]$ is a symbol for the generator corresponding to the face $F$. For $t_1<t_2<\cdots<t_i$
$$d_i([\{m_{t_1}, m_{t_2},\ldots,m_{t_i}\}])=\sum_{j=1}^i(-1)^{j+1}\frac{\lcm(\{m_{t_1},\ldots,m_{t_i}\})}{\lcm(\{m_{t_1},\ldots,\widehat{m_{t_j}} ,\ldots,m_{t_i}\})}[\{m_{t_1},\ldots,\widehat{m_{t_j}} ,\ldots ,m_{t_i}\}]. $$

\begin{example}\label{ex: rooted faces}
Let $\Gamma$ be the simplicial complex given in Figure ~\ref{fig: simplicial tree}. Let $F_1 < F_2 < F_3 < F_4$ be an ordering on the facets of $\Gamma$. Observe that $\min(\{F_2, F_4\})=F_1$, so $\{F_2, F_4\}$ is not a face of $\Lambda_{\mathcal{F}(\Gamma),<}$. However $\{F_1, F_2, F_3\}$ and $\{F_1, F_3, F_4\}$ are rooted and therefore they are the facets of $\Lambda_{\mathcal{F}(\Gamma),<}$. The simplicial complexes $\taylor(\mathcal{F}(\Gamma))$ and $\Lambda_{\mathcal{F}(\Gamma),<}$ are illustrated in Figures~\ref{fig: taylor simplicial complex} and~\ref{fig: lyubeznik simplicial complex} respectively.
\end{example}

\begin{figure}[ht]
\begin{minipage}[b]{0.49\linewidth}
\quad
\quad
\quad
\quad
\begin{tikzpicture}
[scale=1.5, vertices/.style={draw, fill=black, circle, inner sep=0.5pt}]
\filldraw[fill=black!25, draw=black] (2,0.5) -- (2.5,0) -- (2.9,1.5) -- cycle; 
\filldraw[fill=black!25, draw=black] (2.5,0) -- (2.9,1.5) -- (3.8,0.4) -- cycle;
\node[vertices, label=below:{$F_3$}] (a) at (2.5,0) {};
\node[vertices, label=right:{$F_2$}] (b) at (3.8,0.4) {};
\node[vertices, label=left:{$F_4$}] (c) at (2,0.5) {};
\node[vertices, label=above:{$F_1$}] (d) at (2.9,1.5) {};
\foreach \to/\from in {a/b,a/c,a/d,b/d,c/d}
\draw [-] (\to)--(\from);
\draw [dashed] (b)--(c);
\end{tikzpicture}
\caption{$\taylor(\mathcal{F}(\Gamma))$}
\label{fig: taylor simplicial complex}
\end{minipage}
\begin{minipage}[b]{0.47\linewidth}
\quad
\quad
\quad
\quad
\begin{tikzpicture}
[scale=1.5, vertices/.style={draw, fill=black, circle, inner sep=0.5pt}]
\filldraw[fill=black!25, draw=black] (2,0.5) -- (2.5,0) -- (2.9,1.5) -- cycle; 
\filldraw[fill=black!25, draw=black] (2.5,0) -- (2.9,1.5) -- (3.8,0.4) -- cycle;
\node[vertices, label=below:{$F_3$}] (a) at (2.5,0) {};
\node[vertices, label=right:{$F_2$}] (b) at (3.8,0.4) {};
\node[vertices, label=left:{$F_4$}] (c) at (2,0.5) {};
\node[vertices, label=above:{$F_1$}] (d) at (2.9,1.5) {};
\foreach \to/\from in {a/b,a/c,a/d,b/d,c/d}
\draw [-] (\to)--(\from);
\end{tikzpicture}
\caption{$\Lambda_{\mathcal{F}(\Gamma),<}$}
\label{fig: lyubeznik simplicial complex}
\end{minipage}
\end{figure}

We will make use of the following result which was noted by Barile in \cite{barile}.
\begin{theorem}[\cite{barile}]\label{l:nonvanishing condition from Lyubeznik resolution} Let $I$ be a monomial ideal and let $<$ be a total order on the set of minimal generators of $I$. If there exists a facet $F=\{m_{t_1}, m_{t_2},\ldots,m_{t_i}\}$ of $\Lambda_{I,<}$ such that
$$\lcm(m_{t_1},\ldots,\widehat{m_{t_j}},\ldots,m_{t_i})\neq \lcm(m_{t_1},\ldots ,m_{t_i}) $$
for all $j=1,2,\ldots,i$ then $b_{i,\mathbf{q}}(S/I)\neq 0$ where $\mathbf{q} =\lcm(m_{t_1},\ldots,m_{t_i})$.
\end{theorem}

So, for the example above we can see that $\{F_1, F_2, F_3\}$ satisfies the assumptions of Theorem~\ref{l:nonvanishing condition from Lyubeznik resolution} and thus $b_{3,6}(S/\mathcal{F}(\Gamma))\neq 0$.

\section{Resolutions via well ordered facet covers}

\begin{definition}[Well ordered facet cover]A sequence $F_1,\ldots,F_k$ of facets of a simplicial complex $\Gamma$ is called a \textbf{well ordered facet cover} if $\{F_1,\ldots,F_k\}$ is a minimal facet cover of $\Gamma$ and for every facet $H\notin \{F_1,\ldots,F_k\}$ of $\Gamma$ there exists $i\leq k-1$ such that  $F_i \subseteq H \cup F_{i+1} \cup F_{i+2}\cup \cdots \cup F_k.$
\end{definition}

\begin{example}
Let $D$ be an induced matching of $\Gamma$. Then any ordering on the elements of $D$ forms a well ordered facet cover for the induced subcollection $\Gamma_A$ where $A= \displaystyle \cup _{\tiny{F\in D}}F$.
\end{example}

\begin{theorem}\label{thm: well ordered cover is maximal rooted}
Let $F_1,\ldots,F_i$ be a well ordered facet cover of $\Gamma$. Then there is a total order $<$ on the facets of $\Gamma$ such that $\{F_1, \ldots, F_i\}$ is a facet of the Lyubeznik simplicial complex $\Lambda_{\mathcal{F}(\Gamma),<}$.
\end{theorem}
\begin{proof}
First note that since $\{F_1,\ldots,F_i\}$ is a minimal facet cover of $\Gamma$, we have
\begin{equation}\label{eq:minimal facet cover minimality}
F_1\cup \cdots \cup \widehat{F_{\ell}}\cup \cdots \cup F_i \neq F_1\cup \cdots \cup F_i
\end{equation}
for all $\ell=1,\ldots,i$. Consider the order
$$F_1 < F_2 <\cdots < F_i < \facets(\Gamma)\setminus \{F_1,\ldots,F_i\}$$
on the facets of $\Gamma$ where the facets in $ \facets(\Gamma)\setminus \{F_1,\ldots,F_i\}$ have any fixed order. Observe that for any $\{F_{j_1},\ldots , F_{j_t}\} \subseteq \{F_1,\ldots,F_i\}$ with $j_1< \cdots < j_t$ we have $\min(\{F_{j_1},\ldots , F_{j_t}\})=F_{j_1}$ because of the inequality in~\eqref{eq:minimal facet cover minimality}. Therefore $\{F_1,\ldots,F_i\}$ is rooted. To see the maximality, assume for a contradiction that $\{F_1,\ldots,F_i,H\}$ is rooted for some $H \notin \{F_1,\ldots,F_i\}$. But then since $F_1,\ldots,F_i$ is a well ordered facet cover of $\Gamma$ there exists $k\leq i-1$ such that $ F_k \subseteq H \cup F_{k+1} \cup F_{k+2}\cup \cdots \cup F_i$. Thus
$\min (\{H, F_{k+1}, F_{k+2}, \ldots ,F_{i}\}) \leq F_k $ by the given order and $\min(\{H,F_{k+1},\ldots, F_i\}) \notin \{H,F_{k+1},\ldots, F_i\}$. This contradicts rootedness of $\{F_1,\ldots,F_i,H\}$. 
\end{proof}

\begin{corollary}[Betti numbers from facet covers]\label{cor: nonvanishing Betti number sufficient condition}
Let $\Gamma$ be a simplicial complex and let $\m$ be a squarefree monomial. Suppose that $\Gamma_{\m}$ has a well ordered facet cover of cardinality $i$. Then $b_{i,\m}(S/\mathcal{F}(\Gamma))\neq 0$.
\end{corollary}
\begin{proof}
Since any minimal facet cover satisfies~\eqref{eq:minimal facet cover minimality} the proof follows from applying Theorem~\ref{thm: well ordered cover is maximal rooted} to Theorem~\ref{l:nonvanishing condition from Lyubeznik resolution}.
\end{proof}

In Proposition~\ref{generalization verify} we will show that well
ordered edge covers of graphs correspond to strongly disjoint
bouquets. Therefore the corollary above generalizes Proposition 2.5 of
Katzman \cite{katzman} and Theorem 3.1 of Kimura \cite{kimura}.

In \cite[Corollary~3.9]{morey villarreal} Morey and Villarreal gave a
lower bound for regularity of squarefree monomial ideals (see
\eqref{eqn: morey villarreal regularity bound} in
Section~\ref{ss:reg-pd}), improving the previously known bounds by Katzman~\cite[Lemma~2.2]{katzman} and  H\`{a} and Van Tuyl~\cite[Theorem~6.5]{ha van tuyl}. Using
Corollary~\ref{cor: nonvanishing Betti number sufficient condition} we
can further improve the regularity bound mentioned above.

\begin{corollary}[Combinatorial bound for regularity]\label{cor: regularity bound} Let $\Gamma$ be a simplicial complex and let $F_1,\ldots,F_s$ be a well ordered facet cover of some induced subcollection of $\Gamma$. Then
$$\reg(S/\mathcal{F}(\Gamma))\geq \left | \bigcup_{i=1}^s F_i \right |  - s. $$
\end{corollary}

In the case of simplicial forests, the Betti numbers are completely characterized by the well ordered facet covers. We set to prove this next by observing how such facet covers behave under localization.

\begin{lemma}\label{prop: minimal facet cover of localization}Let
  $\Gamma$ be a simplicial complex and let $F$ be a facet of $\Gamma$
  that contains a free vertex. Let $P=(x_i: \ x_i \notin F)$ be an
  ideal of $\Bbbk[x_1,\ldots,x_n]$. Suppose that
  $\mathcal{F}(\Gamma\setminus \langle F\rangle)_{P}$ is the facet
  ideal of $\Delta$ and that $V(\Delta)=V(\Gamma)\setminus F$. Then
  the following hold.

\begin{itemize}
  \setlength{\itemsep}{0pt}

\item[$(1)$]If $\{F_1\setminus F,\ldots,F_k \setminus F \}$ is a minimal facet cover of $\Delta$, then $\{F_1,\ldots,F_k,F\}$ is a minimal facet cover of $\Gamma$.
\item[$(2)$] If $F_1\setminus F,\ldots,F_k \setminus F$ is a well ordered facet cover of $\Delta$, then $F_1,\ldots,F_k,F$ is a well ordered facet cover of $\Gamma$. 
\end{itemize}
\end{lemma}

\begin{proof} 
$(1)$ Assume that $\{F_1\setminus F,\ldots,F_k \setminus F \}$ is a minimal facet cover of $\Delta$. Observe that $\{F_1,\ldots,F_k,F\}$ covers $\Gamma$ since $V(\Delta)=V(\Gamma)\setminus F$. To see the minimality of $\{F_1,\ldots,F_k,F\}$, assume for a contradiction one of its elements is redundant. Note that the redundant facet cannot be $F$ since it contains a free vertex of $\Gamma$. So say $F_s$ is redundant for some $s=1,\ldots,k.$ Then we obtain $F_s\subseteq F\cup (\cup_{i\neq s}F_i)$ or equivalently $F_s\setminus F\subseteq \cup_{i\neq s}(F_i\setminus F)$ which contradicts the minimality of $\{F_1\setminus F,\ldots,F_k \setminus F \}$. 

$(2)$ Assume that $F_1\setminus F,\ldots,F_k \setminus F$ is a well ordered facet cover of $\Delta$. By part $(1)$, $\{F_1,\ldots,F_k,F\}$ is a minimal facet cover of $\Gamma$. Let $H\notin \{F_1,\ldots,F_k,F\}$ be a facet of $\Gamma$. Then we consider two cases:

\begin{case-}If $H\setminus F$ is a facet of $\Delta$, then by assumption
$$F_s\setminus F \subseteq (H\setminus F) \cup (F_{s+1}\setminus F) \cup \cdots \cup(F_{k}\setminus F)$$ for some $s\leq k-1$ and hence $F_s \subseteq H \cup F_{s+1} \cup \cdots \cup F_{k}\cup F$ as desired.
\end{case-}
\begin{case-} If $H \setminus F$ is not a facet of $\Delta$, then
  $K\setminus F \subseteq H \setminus F $ for some facet $K\setminus
  F$ of $\Delta$. Now if $K\setminus F =F_t \setminus F$ for some
  $1\leq t \leq k$, then $$F_t \setminus F \subseteq H\setminus F
  \mbox{ and } F_t \subseteq H \cup F \subseteq H \cup F_{t+1} \cup
  \cdots \cup F_{k} \cup F$$ as desired. Therefore we assume that
  $K\setminus F \neq F_t \setminus F$ for all $1\leq t \leq k$. Then
  by assumption we have $$F_{\ell}\setminus F \subseteq (K\setminus F)
  \cup (F_{\ell +1}\setminus F) \cup \cdots \cup (F_{k}\setminus F)$$
  for some $\ell\leq k-1$. Thus we get $F_{\ell} \subseteq K \cup
  F_{\ell +1} \cup \cdots \cup F_{k} \cup F \subseteq H \cup F_{\ell
    +1} \cup \cdots \cup F_{k} \cup F $ which completes the proof.
\end{case-}\end{proof}


For the class of simplicial forests, existence of well ordered facet covers characterizes Betti numbers:
\begin{theorem}[Combinatorial description for Betti numbers of simplicial forests]\label{thm: characterization of Betti numbers}
Let $\Gamma$ be a simplicial forest. Suppose that $\m$ is a monomial and $i\geq 1$. Then the following are equivalent.
\begin{itemize}
  \setlength{\itemsep}{0pt}

\item[$(1)$] $b_{i,\m}(S/\mathcal{F}(\Gamma))\neq 0$.
\item[$(2)$] $b_{i, \m}(S/\mathcal{F}(\Gamma))=1$.
\item[$(3)$] The induced subcollection $\Gamma_{\m}$ has a well ordered facet cover of cardinality $i$.
\end{itemize}
In particular, $b_{i,j}(S/\mathcal{F}(\Gamma))$ is the number of induced subcollections of $\Gamma$  which have $j$ vertices and which have well ordered facet covers of cardinality $i$. 
\end{theorem}

\begin{proof}
  The equivalence $(1) \iff (2)$ follows from Theorem~\ref{thm:
    multigraded betti numbers of simplicial forests}. So we only need
  to prove $(1) \iff (3)$. We may assume that $i\geq 2$ since the
  statement is trivial for $i=1.$ By \cite[Lemma 3.1]{erey faridi}, it
  suffices to prove that $b_{i, n}(S/\mathcal{F}(\Gamma))\neq 0$ if
  and only if $\Gamma$ has a well ordered facet cover of cardinality
  $i$ where $n$ is the number of vertices of $\Gamma$. 

  First observe that the implication $(3) \Longrightarrow (1)$ follows
  from Corollary~\ref{cor: nonvanishing Betti number sufficient
    condition}. So we only need to prove $(1) \Longrightarrow (3)$. 
  Suppose that $b_{i,n}(S/\mathcal{F}(\Gamma))\neq 0$. We will proceed
  by induction on the number of vertices of $\Gamma$. If $\Gamma$ is
  not connected, then we can apply \cite[Lemma 3.2]{erey faridi} to
  $\mathcal{F}(\Upsilon_1),\ldots,\mathcal{F}(\Upsilon_k)$ where
  $\Upsilon_1, \ldots , \Upsilon_k$ are the connected components of
  $\Gamma$. Then by using Theorem~\ref{thm: multigraded betti numbers
    of simplicial forests} we get $$0\neq b_{i,
    n}(S/\mathcal{F}(\Gamma))=b_{u_1,q_1}(S/\mathcal{F}(\Upsilon_1))\cdots
  b_{u_k,q_k}(S/\mathcal{F}(\Upsilon_k))$$ for some $u_1, \ldots ,u_k$
  such that $u_1+\cdots+u_k=i$, where $q_1,\ldots,q_k$ are the number
  of vertices of $\Upsilon_1,\ldots,\Upsilon_k$, respectively. But
  then by induction hypothesis for each $t=1,\ldots,k$ the simplicial
  tree $\Upsilon_t$ has a well ordered facet cover $F^t_1, \ldots
  ,F^t_{u_t}$. Hence $F^1_1, \ldots ,F^1_{u_1}, \ldots , F^k_1 ,
  \ldots , F^k_{u_k}$ is a well ordered facet cover of $\Gamma$ of
  cardinality $i$.

Now we may assume that $\Gamma$ is a simplicial tree on the vertices
$x_1,\ldots,x_n$. Let $F$ be a leaf of $\Gamma$ and
$\mathcal{F}(\Gamma\setminus \langle F \rangle)_{(x_i: \ x_i \notin
  F)}$ be the facet ideal of $\Delta$. Observe that $\Delta$ has at
most $n-|F|$ vertices. By Equation~\eqref{eq:
    recursive multigraded formula} we have $0 \neq 
b_{i,n}(S/\mathcal{F}(\Gamma))=b_{i-1,
  n-|F|}(S/\mathcal{F}(\Delta))$. Hence the nonvanishing Betti number
and \cite[Remark 2.4]{erey faridi} require $\Delta$ to have exactly
$n-|F|$ vertices. By induction hypothesis $\Delta$ has a well ordered
facet cover of cardinality $i-1$. Thus Lemma~\ref{prop: minimal facet cover
  of localization} gives that $\Gamma$ has a well ordered facet cover
of cardinality $i$ which completes the proof.
\end{proof}


\begin{example}\label{ex: computing betti numbers}
Suppose that $\Gamma$ is the simplicial tree in Figure~\ref{fig:
  simplicial tree}. We wish to apply Theorem~\ref{thm:
  characterization of Betti numbers} to find the Betti numbers of
$S/\mathcal{F}(\Gamma)$. Observe that $F_1, F_2, F_3$ is a well
ordered facet cover of $\Gamma$ since $F_1\subseteq F_2\cup F_3 \cup
F_4$. Therefore $b_{3,6}=1$. Since the multigraded Betti numbers come
from the induced subcollections we check which subcollections give
Betti numbers. We see that $\Gamma$ has two induced subcollections
which are generated by $3$ facets, namely
$\Gamma_{x_1x_2x_3x_4x_5}=\langle F_1, F_3, F_4 \rangle$ and
$\Gamma_{x_1x_3x_4x_5x_6}=\langle F_1, F_2, F_4 \rangle$. These two
simplicial complexes are isomorphic and have no well ordered facet
covers, so they do not give Betti numbers. Next, we see that $\Gamma$
has four induced subcollections which are generated by $2$ facets,
namely $\Gamma_{x_3x_4x_5x_6}=\langle F_1, F_2 \rangle$,
$\Gamma_{x_1x_3x_4x_5}=\langle F_1, F_4 \rangle$,
$\Gamma_{x_1x_2x_3x_5x_6}=\langle F_2, F_3 \rangle$ and
$\Gamma_{x_1x_2x_3x_4}=\langle F_3, F_4 \rangle$. For each of these
subcollections the facet set is the same as the unique minimal facet
cover. Therefore they all have well ordered facet covers of
cardinality $2$ and
$b_{2,x_3x_4x_5x_6}=b_{2,x_1x_3x_4x_5}=b_{2,x_1x_2x_3x_5x_6}=b_{2,x_1x_2x_3x_4}=1$. Finally
every facet of $\Gamma$ generates an induced subcollection with a well
ordered facet cover. Thus
$b_{1,x_1x_2x_3}=b_{1,x_1x_3x_4}=b_{1,x_3x_4x_5}=b_{1,x_3x_5x_6}=1$. The
Betti diagram of $S/\mathcal{F}(\Gamma)$ is
\end{example}

\begin{center}
\begin{tabular}{r| c c c c c }

 & 0 & 1 & 2 & 3 \\
 \hline
Total & 1 & 4 & 4 & 1\\
\hline
0 & 1 & -- & -- & --\\

1 & -- & -- & -- & -- \\
2 & -- & 4 & 3 & --\\
3 & -- & -- & 1 & 1\\
\end{tabular}
\end{center}
where $i$th column and $j$th row is $b_{i,i+j}(S/\mathcal{F}(\Gamma))$.

\begin{remark}
Although the Betti numbers of facet ideals of simplicial forests can be described as in Theorem~\ref{thm: characterization of Betti numbers}, it is not possible in general to minimally resolve such ideals by Lyubeznik resolution. In fact, the ideal given in Example~\ref{ex: computing betti numbers} has no simplicial resolution. To see this, assume for a contradiction that $\Theta$ supports a minimal free resolution of $S/\mathcal{F}(\Gamma)$. Looking at the multigraded Betti number in the third homological degree, we can see that either $\{F_2, F_3, F_4\}$ or $\{F_1, F_2, F_3\}$ must be a face of $\Theta$. If $\{F_2, F_3, F_4\}$ is a face, then $\{F_2, F_4\}$ is a face as well. But then $b_{2,x_1x_3x_4x_5x_6}\neq 0$ which is not true. Similarly if $\{F_1, F_2, F_3\}$ is a face, then $\{F_1, F_3\}$ is a face and we get $b_{2,x_1x_2x_3x_4x_5}\neq 0$, a contradiction.

\end{remark}

\subsection{Regularity and projective dimension}\label{ss:reg-pd}

Katzman~\cite{katzman} proved that when $G$ is a graph with edge ideal
$I(G)$ (i.e. the facet ideal of $G$ where $G$ is
considered a $1$-dimensional simplicial complex), then the
regularity of $S/I(G)$ is bounded below by the induced matching number
of $G$. Zheng \cite{zheng} was the first one who showed that such
bound is sharp if $I(G)$ is the edge ideal of a graph forest. In fact,
many interesting graph families, including chordal graphs \cite{ha van
  tuyl} and very well-covered graphs \cite{mahmoudi et al} are known to have the regularity of $S/I(G)$ equal to the
induced matching number of $G$. In \cite{morey villarreal} Morey and
Villarreal showed that if $\Gamma$ is a simplicial complex with facet
ideal $\mathcal{F}(\Gamma)$, then
\begin{equation}\label{eqn: morey villarreal regularity bound}
\reg(S/\mathcal{F}(\Gamma))\geq \max \Bigg\{ \left | \bigcup_{i=1}^s F_i \right | - s \mid \{F_1,\ldots, F_s\} \text{ is an induced matching in } \Gamma \Bigg\}
\end{equation}
which extends Katzman's bound to simplicial complexes. 

Since simplicial forests are higher dimensional analogues of graph
forests, one may expect that their facet ideals attains the bound
given above as in the case of edge ideals of graph forests. However
one can find examples of simplicial trees for which the regularity is
arbitrarily bigger than such bound, see for instance
\cite[Example~4.11]{ha regularity}.

Using well ordered facet covers we are able to improve the bound given
in~\eqref{eqn: morey villarreal regularity bound} and
express the regularity of facet ideals of simplicial forests in terms
of well ordered facet covers. Prior to our characterization no
combinatorial formula was known for the regularity of facet ideals of
simplicial forests.

\begin{corollary}[Combinatorial description for projective dimension and regularity of simplicial forests]\label{c:reg-pd} If $\Gamma$ is a simplicial forest, then
\begin{itemize}
  \setlength{\itemsep}{0pt}

\item[$(1)$] $\pd(S/\mathcal{F}(\Gamma))$ is the maximum cardinality of a well ordered facet cover of an induced subcollection of $\Gamma$.
\item[$(2)$] $\reg(S/\mathcal{F}(\Gamma))$ is equal to
\begin{equation*} \displaystyle
\max \big\{\deg(\m) - s_m \mid  F_1,\ldots, F_{s_m} \text{ is a well ordered facet cover of } \Gamma_{\m} \big\}.
\end{equation*}
\item[$(3)$] All well ordered facet covers of a simplicial forest have the same cardinality. 
\end{itemize}
\end{corollary}

\begin{proof}
Immediately follows from Theorem~\ref{thm: characterization of Betti numbers} and Theorem~\ref{thm: multigraded betti numbers of simplicial forests}.
\end{proof}

\begin{remark}
  Morey and Villarreal \cite[Corollary~3.33]{morey villarreal} proved
  that the projective dimension of sequentially Cohen-Macaulay
  squarefree monomial ideals could be characterized in terms of
  minimal vertex covers. This implies by \cite[Corollary~5.6]{faridi
    sequentially cm} that the facet ideals of simplicial forests are
  known to have a closed formula for their projective
  dimension. However, our formula in Corollary~\ref{c:reg-pd} is different since it is expressed in terms of minimal
  facet covers instead of minimal vertex covers.
\end{remark}

In \cite{bouchat ha okeefe} Bouchat, H\`{a} and A. O'Keefe studied
path ideals of rooted trees. Using the mapping cone construction they
obtained numerical formulas for the invariants of such ideals. Since
the path ideal of a rooted tree is the facet ideal of a simplicial
tree \cite[Corollary 2.9]{he van tuyl}, our results provide a new
combinatorial method to study path ideals of rooted trees. Note that
not every facet ideal of a simplicial tree is the path ideal of a
rooted tree. Therefore our approach is more general in this setting.

\section{Well ordered edge covers of graphs}
In this section we will show that well ordered edge covers of graphs correspond to certain bouquet (star) subgraphs.

Let $G$ be a finite simple graph with the vertex set $V(G)$ and the edge set $E(G)$. A \textbf{bouquet} is a graph $B$ with $V(B)=\{r,z_1,\ldots,z_d\}$ and $E(G)=\{\{r,z_i\} \mid i=1,\ldots,d \}$ where $d\geq 1$. 

We say that $H$ is a\textbf{ subgraph} of $G$ if the vertex set and edge set of $H$ are contained in those of $G$. If a bouquet $B$ is a subgraph of $G$, for simplicity we say that $B$ is a \textbf{bouquet of} $G$. Let $\bm{\mathcal{B}}=\{B_1,\ldots,B_q\}$ be a set of bouquets of $G$. We set $E(\bm{\mathcal{B}})= \cup_{i=1}^qE(B_i) $ and $V(\bm{\mathcal{B}})=\cup_{i=1}^qV(B_i)$.

\begin{definition}[{\cite[Definitions~2.1, 2.3]{kimura}}]
A set $\bm{\mathcal{B}}=\{B_1,\ldots,B_q\}$ of bouquets of $G$ is called \textbf{disjoint} in $G$ if $V(B_k)\cap V(B_{\ell})=\emptyset$ for all $k\neq \ell$. Moreover, if for every $k=1,\ldots ,q$ there exists $e_k\in E(B_k)$ such that $\{e_1,\ldots,e_q\}$ is an induced matching in $G$, then $\bm{\mathcal{B}}$ is called a \textbf{strongly disjoint} set of bouquets in $G$. We say that $G$ \textbf{contains a strongly disjoint set of bouquets} if there exists a strongly disjoint set of bouquets $\bm{\mathcal{B}}$ of $G$ such that $V(G)=V(\bm{\mathcal{B}})$.
\end{definition}
\begin{remark}\label{edge cover of graphs}
It is a well known fact in graph theory that if $D$ is a minimal edge
cover of a graph $G$, then there is a set $\bm{\mathcal{B}}$ of
disjoint bouquets in $G$ such that $E(\bm{\mathcal{B}})=D$. We refer
to~\cite[Theorem 3.1.22]{west} for a proof of this fact.
\end{remark}

\begin{proposition}\label{generalization verify}
  Let $G$ be a graph  and let
  $d_1,\ldots ,d_n \in E(G)$. Then the following are equivalent.
\begin{itemize}
\item[$(1)$] Some permutation of $d_1,\ldots ,d_n$ is a well ordered
  edge cover for $G$.
\item[$(2)$] $G$ contains a strongly disjoint set of bouquets $\bm{\mathcal{B}}$ with $E(\bm{\mathcal{B})}=\{d_1, \ldots, d_n\}$.
\end{itemize}
\end{proposition}

\begin{proof}
  First suppose that $G$ has a well ordered edge cover $d_1, \ldots ,
  d_n$. Then by Remark~\ref{edge cover of graphs} there exists a set
  $\bm{\mathcal{T}}=\{T_1,\ldots,T_q\}$ of disjoint bouquets with
  $E(\bm{\mathcal{T}})=\{d_1, \ldots , d_n\}$. We claim that
  $\bm{\mathcal{T}}$ is strongly disjoint. Let
  $s_p=d_{\max\{i \mid d_i \in E(T_p)\}}$ for every
  $p=1,\ldots,q$. We will show that $\{s_1,\ldots,s_q\}$ is an induced
  matching. Let $h\notin \{d_1, \ldots , d_n\}$ be an edge of
  $G$. Then by definition of well ordered edge covers there is some
  $\ell \leq n-1$ for which $d_{\ell} \subseteq h\cup d_{\ell +1}\cup
  \cdots \cup d_n$. Since each edge has two vertices, we
    have $d_\ell\subseteq h\cup d_k$ for some $k>\ell$. This implies
    that $d_k\cap d_{\ell}\neq \emptyset$. Therefore $d_\ell$ and
    $d_k$ belong to the same bouquet of $\bm{\mathcal{T}}$, say
    $T_m$. Let $u$ be the vertex of $d_{\ell}$ which does not belong
    to $d_k$, so $u$ is not in any other  edge of
      $T_m$. In particular, as $k >\ell$, we have that $s_m \neq
      d_{\ell}$ and therefore $u \notin s_m$. Since $d_\ell\subseteq
      h\cup d_k$, the vertex $u$ is also in $h$. Therefore $h$
      contains a vertex, namely $u$, which does not belong to $s_p$
      for each $p=1,\ldots,q$. Then we conclude that there is no pair
      $i,j$ $(i\neq j)$ such that $h$ intersects both of $s_i$ and
      $s_j$.

  Conversely, suppose that $G$ contains a strongly disjoint set of
  bouquets $\bm{\mathcal{B}}=\{B_1,\ldots,B_q\}$ where
  $E(B_p)=\{e^p_1,\ldots,e^p_{t_p},s_p\}$ for every $p=1,\ldots,q$ and
  $\{s_1,\ldots ,s_q\}$ is an induced matching in $G$. It is clear that
  $E(\bm{\mathcal{B})}$ is a minimal edge cover of $G$. We claim that
  $$E(\bm{\mathcal{B}})\setminus \{s_1,\ldots,s_q\} , s_1, \ldots , s_q$$ is a well
  ordered edge cover of $G$ where the edges in $E(\bm{\mathcal{B}})\setminus
  \{s_1,\ldots,s_q\}$ are listed in any fixed order.

Let $h\in E(G)\setminus E(\bm{\mathcal{B})} $. Observe that since $\{s_1, \ldots , s_q\}$ is an induced matching, $h \cap (\cup_{r=1}^q s_r)$ has cardinality at most one. Therefore $h$ contains at least one vertex which do not belong to $\cup_{r=1}^q s_r$. Then there exists $p\in\{1,\ldots , q\}$ and $1\leq j \leq t_p$ such that $e_j^p\cap h\neq \emptyset$ and $e_j^p\cap h \neq e_j^p \cap s_p$. Hence $e^p_j \subseteq h \cup s_p$ as desired. 
\end{proof}

\end{document}